\documentclass[10pt]{amsart}

 \usepackage{amsmath,amsthm}
\usepackage{amssymb,latexsym}

\makeatletter
\@addtoreset{equation}{section}
\makeatother

\newtheorem{theorem}{Theorem}[section]
\newtheorem{proposition}[theorem]{Proposition}
\newtheorem{lemma}[theorem]{Lemma}
\newtheorem{corollary}[theorem]{Corollary}

\theoremstyle{definition}

\newcommand{\what}[1]{\widehat{#1}}

\newcommand{\til}[1]{\tilde{#1}}

\newcommand{\id}{\operatorname{id}}

\newcommand{\norm}[1]{\left\Vert#1\right\Vert}
\newcommand{\pcbnorm}[1]{\left\Vert#1\right\Vert_{p\mathrm{cb}}}

\newcommand{\fA}{\mathcal{A}}
\newcommand{\fB}{\mathcal{B}}
\newcommand{\fC}{\mathcal{C}}
\newcommand{\fF}{\mathcal{F}}

\newcommand{\fN}{\mathcal{N}}
\newcommand{\fQ}{\mathcal{Q}}
\newcommand{\fS}{\mathcal{S}}
\newcommand{\fV}{\mathcal{V}}
\newcommand{\fW}{\mathcal{W}}
\newcommand{\fX}{\mathcal{X}}
\newcommand{\fY}{\mathcal{Y}}
\newcommand{\fZ}{\mathcal{Z}}

\newcommand{\Cee}{\mathbb{C}}

\newcommand{\En}{\mathbb{N}}

\newcommand{\alp}{\alpha}
\newcommand{\eps}{\varepsilon}
\newcommand{\gam}{\gamma}
\newcommand{\lam}{\lambda}
\newcommand{\ome}{\omega}
\newcommand{\Ome}{\Omega}

\newcommand{\mat}{\mathbb{M}}

\begin{document}

\title[Min.\ and max.\ $p$-operator spaces]
{On minimal and maximal $p$-operator space structures}

\author{Serap \"{O}ztop and Nico Spronk}

\begin{abstract}
We show that $L^\infty(\mu)$, in its capacity as multiplication operators on $L^p(\mu)$,
is minimal as a $p$-operator space for a decomposable measure $\mu$.  
We conclude that $L^1(\mu)$ has a certain maximal type $p$-operator space structure
which facilitates computations with $L^1(\mu)$ and the projective tensor product.
\end{abstract}

\maketitle

\footnote{{\it Date}: \today.

2000 {\it Mathematics Subject Classification.} Primary 46L07;
Secondary 47L25, 46G10.
{\it Key words and phrases.} $p$-operator space, min space, max space.

The first named author would like to thank the Scientific Projects Coordination Unit
of the Istanbul University, IRP 11488,  and the 
University of Waterloo for hosting her visit fromApril 2011 to June 2012.  
The second named author would
like thank NSERC Grant 312515-2010.}

%%Main Body

%write breif intro here

In the theory of operator spaces, there are extremal operator space structures which
can be assigned to any Banach space.  These arose in the papers \cite{blecherpaulsen,effrosruan1}
and are exposed in the monograph \cite{effrosruan}.  They have particular
value when understanding mappings and tensor products.

In the present article we examine minimal and maximal $p$-operator space structures.
These structures' existences were noted in \cite{junge}, where they were used
to characterise certain algebras as algebras of operators on $\fS\fQ_p$-spaces.
Our primary motivation is to gain the isometric tensor product formula 
$L^1(\mu)\what{\otimes}^p\fV\cong L^1(\mu,\fV)$ for the $p$-operator projective tensor
product of \cite{daws}.   Here $L^1(\mu)$ has
a certain maximal operator space structure, which appears naturally via the embedding
of $L^1(\mu)\hookrightarrow L^\infty(\mu)^*$, where $L^\infty(\mu)$ acts on
$L^p(\mu)$ as multiplication operators.  This is a less obvious task then we had initially
hoped, and seems worth an exposition in its own right.  
The techniques of this article are all classical and elementary.

\subsection{Background}
Let $1<p<\infty$, and $p'$ denote the conjugate index so $\frac{1}{p}+\frac{1}{p'}=1$.  
The theory of $p$-operator spaces is designed to give an analogue
to the theory of operator spaces on a Hilbert space, which we might call $2$-operator spaces.
The theory of $p$-operator spaces has its origins in \cite{pisier,lemerdy}, and was studied extensively
in \cite{junge}.  Daws~(\cite{daws}) presents these spaces in the format we are using, a format also
used extensively by An, Lee and Ruan~(\cite{anleeruan}).  We closely follow the presentation of
\cite{daws} and use some concepts from \cite{anleeruan}.

We let $\ell^p(n)$ denote $\Cee^n$ with the $\ell^p$-norm.
Given a Banach space $\fV$, a $p$-operator space structure on $\fV$ is a sequence
of norms $\norm{\cdot}_n$, each norm on $n\times n$-matricies with entries in
$\fV$, which satisfy the axioms below.

$(D_\infty)$ \phantom{m} \parbox[t]{4.3in}{For $u$ in $\mat_n(\fV)$ and $v$ in $\mat_m(\fV)$,
$\norm{u\oplus v}_{n+m}=\max\{\norm{u}_n,\norm{v}_m\}$.}

$(M_p)$ \phantom{mj} \parbox[t]{4.3in}{For $u$ in $\mat_n(\fV)$ and
$\alp,\beta$ in $\mat_n\cong\fB(\ell^p(n))$, $\norm{\alp u\beta}_n\leq
\norm{\alp}_{\fB(\ell^p(n))}\norm{u}_n\norm{\beta}_{\fB(\ell^p(n))}$.}

\noindent A Banach space $\fV$, equipped with a sequence of norms as above,
will be called a {\em $p$-operator space}.  In the sequel we will drop the subscript
$n$ from the norm on $\mat_n(\fV)$.  A linear map $T:\fV\to\fW$ gives rise to 
amplifications $T^{(n)}:\mat_n(\fV)\to\mat_n(\fW)$, $T^{(n)}[v_{ij}]=[Tv_{ij}]$.
Such a map is called {\em completely bounded} if $\pcbnorm{T}
=\sup_n\norm{T^{(n)}}<\infty$.  Moreover it is called {\em completely contractive}
if $\pcbnorm{T}\leq 1$ and a {\em complete isometry} if each $T^{(n)}$ is an isometry.
The space of such maps will be denoted $\fC\fB_p(\fV,\fW)$.

We say a Banach space $E$ is in the class $\fS\fQ_p$ if it is a quotient of a subspace of
$L^p(\phi)$ for some measure $\phi$.  The space $\fB(E)$ is a $p$-operator
space given identifications $\mat_n(\fB(E))\cong\fB(\ell^p(n)\otimes^pE)
\cong\fB(\ell^p(n,E))$.  Here $L^p(\phi)\otimes^pE$ is the completion with respect to the norm 
given by embedding $L^p(\phi)\otimes E\hookrightarrow L^p(\phi,E)$.  Moreover,
any $p$-operator space admits a complete isometry into $\fB(E)$ for some $E$
in $\fS\fQ_p$ (\cite{pisier,lemerdy}).  Spaces which admit complete isometries
into $\fB(L^p(\phi))$ will admit better properies than general $p$-operator spaces.
We will follow \cite{anleeruan} and say that such spaces {\em act on (some) $L^p$}.

We follow \cite{daws} on assigning $p$-operator space structures to mapping spaces.
We identify $\mat_n(\fC\fB_p(\fV,\fW))\cong\fC\fB_p(\fV,\mat_n(\fW))$, where
$\mat_n(\fW)$ is a $p$-operator space via the identifications $\mat_m(\mat_n(\fW))
\cong\mat_{mn}(\fW)$.  In particular, for the dual space, $\mat_n(\fV^*)
\cong\fC\fB_p(\fV,\fB(\ell^p(n))$, completely isometrically.  
We have $p$-version of the projective tensor product
$\otimes^\gam$ and the injective tensor product $\otimes^\lambda$; namely
the {\em $p$-projective tensor product} $\what{\otimes}^p$ of \cite{daws} and
the {\em $p$-injective tensor product} $\check{\otimes}^p$ of \cite{anleeruan}.
The $p$-projective tensor product enjoys all of the usual functorial properties
which are analogues of $\otimes^\gam$, while the theory of $\check{\otimes}^p$
is not as well understood.  However, we do have that 
$\mat_n(\fV)\cong\fV\check{\otimes}^p\fB(\ell^p(n))$ completely isometrically.

As observed in \cite[p.\ 89]{junge}, 
for a $p$-operator space $\fV$, the algebraic idenification 
$\fV\otimes\fB(\ell^p(n))\cong\mat_n(\fV)$
allows us to view $\norm{\cdot}_n$ as a reasonable cross-norm on $\fV\otimes\fB(\ell^p(n))$; see the
terminology in \cite{ryan}, for example.  Indeed,  an application of ($M_p$) then of ($D_\infty$) 
shows that $\norm{[\alp_{ij} v]}_n\leq\norm{\alp}_{\fB(\ell^p(n))}\norm{v}$ for $\alp$ in $\mat_n$
and $v$ in $\fV$; while \cite[Lem.\ 4.2]{daws} --- that contractive linear functions are
automatically completely contractive --- shows that $|\varphi\otimes\psi(v)|\leq
\norm{\varphi}_{\fV^*}\norm{\psi} _{\fB(\ell^p(n))^*}\norm{v}_n$ for any $\varphi$ and $\psi$
where $v\in\fV\otimes\fB(\ell^p(n))$.
Moroever, if $\fX$ is any Banach space, then the algebraic identifications
\begin{equation}\label{eq:minstructure}
\mat_n(\fX)\cong\fX\otimes^\lam\fB(\ell^p(n))
\end{equation}
(injective tensor product on the right),
are easily verified to produce a $p$-operator space structure on $\fX$ which is minimal
in the sense that $\norm{\cdot}_n\leq\norm{\cdot}_n'$ (for each $n$) with any other operator space
structure on $\fX$.  We call this operator space structure the {\em minimimal $p$-operator
space} structure on $\fX$.  If $\fV$ is an operator space and $T:\fV\to\fX$ is bounded then
$T$ is completley bounded with $\pcbnorm{T}=\norm{T}$. Then, by uniformity of the injective 
tensor product, we see that
\[
T^{(n)}\cong (T\otimes\id)\circ\iota_n:\fV\check{\otimes}^p\fB(\ell^p(n))\to
\fV\otimes^\lam\fB(\ell^p(n))\to\fX\otimes^\lam\fB(\ell^p(n))
\]
is bounded with norm at most $\norm{T}$, where $\iota_n$ is the identity on
$\fV\otimes\fB(\ell^p(n))$, which is a contraction as $\check{\otimes}^p$ gives a reasonable 
cross-norm.  We call any $p$-operator space $\fV$ whose 
$p$-operator structure is the minimal one, i.e.\ $\fV=\min\fV$ completely isometrically,
a {\em minimal $p$-operator space}.  

\begin{proposition}\label{prop:minspace}
The following are equivalent for a $p$-operator space $\fV$:

{\bf (i)} $\fV$ is minimal;

{\bf (ii)} for any $p$-operator space $\fW$, $\fC\fB_p(\fW,\fV)=\fB(\fW,\fV)$ isometrically;

{\bf (iii)} for any $p$-operator space $\fW$, $\fW\check{\otimes}^p\fV=\fW\otimes^\lambda\fV$.
\end{proposition}

\begin{proof}
Since $\fC\fB_p(\fW,\fV)\subset\fB(\fW,\fV)$ contractivley, the observation above gives
that (i) implies (ii).  Condition (ii) implies that
$\id:\min\fV\to\fV$ is completley contractive.  Since the converse is automatic, (i) holds.

If (ii) holds then $\fC\fB_p(\fW^*,\fV)=\fB(\fW^*,\fV)$ isometrically.  Thus,
by virtue of the definition of the $p$-operator injective tensor product
(\cite[\S 3]{anleeruan}) and the well known injection $\fW\otimes\fV\hookrightarrow
\fB(\fW^*,\fV)$, the $p$-operator injective and injective tensor norms agree on $\fW\otimes\fV$.
\end{proof}

\noindent The definition of maximal $p$-operator space will be given in Section \ref{sec:maximal}.

The following rudimentary fact will be referred to a couple of times in the sequel,
and is an obvious consequence of the density of simple functions in $L^{p'}(\phi)$
and duality.

\begin{lemma}\label{lem:ellpinLp}
For any finite subset $F\subset L^p(\phi)$ and $\eps>0$, there is an $m$ in $\En$ and a 
contraction $V:L^p(\phi)\to \ell^p(m)$
for which $(1-\eps)\norm{f}_{L^p}\leq\norm{Vf}_{\ell^p}\leq(1+\eps)\norm{f}_{L^p}$
for $f$ in $F$.
\end{lemma}

\section{On minimal $p$-operator spaces}

In the theory of $2$-operator spaces, a special role is played by
commutative C*-algebras and completely isometric copies of their subspaces.  
These are the {\em minimal operator spaces}.  Classical theory tells us that
any representation of a commutative C*-algebra $\fA\cong\fC_0(\Ome)$
on a Hilbert space can be realised as a direct sum of representations on cylic subspaces,
where each, in turn, produces a Radon measure $\nu$ on $\Ome$ by which the representation
is unitarily equivalent to a representation by multiplication operators on $L^2(\nu)$.
We are not aware of any analogue of this result for representation on $\fS\fQ_p$-spaces,
or even $L^p$-spaces.  This reduces us to studying representations which are already
multiplication representations on $L^p$-spaces.  This gives rise to a 
more robust theory than might be anticipated.

%Let us fix $1< p<\infty$.

\subsection{On the space of continuous functions as a minimal $p$-operator space}
We begin with the continuous bounded functions $\fC_b(\Ome)$ on a locally compact space
$\Ome$.   In this case a familiar formula for the injective tensor product gives for each $n$
an isometric identification
\begin{equation}\label{eq:cbf}
\mat_n(\min\fC_b(\Ome))\cong \fC_b(\Ome,\fB(\ell^p(n)),\quad
[f_{ij}]\mapsto(\ome\mapsto[f_{ij}(\ome)]).
\end{equation}
Indeed, the Stone-\v{C}ech compactification satisfies
$\fC(\beta\Ome,M)\cong\fC_b(\Ome,M)$ for any finite dimensional Banach space $M$.
We let $\nu$ be a Radon measure on $\Ome$ and $M_\nu:\fC_b(\Ome)\to\fB(L^p(\nu))$ be the
contractive injection given by 
\[
M_\nu(f)\xi(\ome)=f(\ome)\xi(\ome)
\]
for $\nu$-a.e.\ $\ome$.  We say $\nu$ if faithful is $\nu(U)>0$ for any open set $U$.
If $\nu$ is faithful then $M_\nu$ is an isometry.  

The next simple result is required for the next section.  The result seems like
it ought to hold for more general $L^\infty$-spaces, except for a certain
localisation of norm arguement at the end of the proof.

\begin{proposition}\label{prop:cbf}
Given a faithful Radon measure $\nu$ on $\Ome$, $M_\nu:\min\fC_b(\Ome)\to
\fB(L^p(\nu))$ is a complete isometry.
\end{proposition}

\begin{proof}
It suffices to verify that each amplification $M_\nu^{(n)}$ is an isometry.
We identify $\mat_n(\fB(L^p(\nu)))\cong
\fB(L^p(\nu,\ell^p(n)))$.  We observe, under this identification, that
$M_\nu^{(n)}(F)\xi(\ome)=F(\ome)\xi(\ome)$,
for $F$ in $\fC_b(\Ome,\fB(\ell^p(n)))$, $\xi$ in $L^p(\nu,\ell^p(n))$
and $\nu$-a.e.\ $\ome$.  We compute
\begin{align*}
\norm{M_\nu^{(n)}(F)\xi}_{L^p(\nu,\ell^p)}
&=\left(\int_\Ome\norm{F(\ome)\xi(\ome)}_{\ell^p}^p\,d\nu(\ome)\right)^{1/p} \\
&\leq \left(\int_\Ome\norm{F(\ome)}_{\fB(\ell^p)}^p
\norm{\xi(\ome)}_{\ell^p}^p\,d\nu(\ome)\right)^{1/p} \\
&\leq \norm{F}_{\fC_b(\Ome,\fB(\ell^p))}\norm{\xi}_{L^p(\nu,\ell^p)}
\end{align*}
Thus $M^{(n)}_\nu$ is a contraction.

Conversely, given $\eps>0$, find $\ome_0$ for which
$\norm{F(\ome_0)}_{\fB(\ell^p)}>\norm{F}_{\fC_b(\Ome,\fB(\ell^p))}-\eps$,
and then $\xi_0$ in $\ell^p(n)$ with $\norm{\xi_0}_{\ell^p}=1$ and
for which $\norm{F(\ome_0)\xi_0}_{\ell^p}
=\norm{F(\ome_0)}_{\fB(\ell^p)}$.  Find a compact
neighbourhood $K$ of $\ome_0$
such that $\norm{F(\ome)-F(\ome_0)}_{\fB(\ell^p)}<\eps$ for $\ome$ in $K$.
(This is the ``localisation of norm argument" to which we alluded, above.)
Then $\xi=\nu(K)^{-1/p}1_K(\cdot)\xi_0$
in $L^p(\nu,\ell^p(n))$ is of norm $1$ and
satisfies 
\[
\norm{M^{(n)}_\nu(F)\xi-F(\ome_0)\xi}_{L^p(\nu,\ell^p)}<\eps.
\]
It is immediate that $M^{(n)}_\nu$ is an isometry.
\end{proof}

Of course, the above result applies to $\ell^\infty(\Ome)$ for any set $\Ome$.
Let $\fX$ be a Banach space.  We let 
$\Ome$ denote any subset of the unit ball of $\fX^*$
which is norming for $\fX$, and consider the
isometric embedding
\begin{equation}\label{eq:minembedding}
\fX\hookrightarrow\ell^\infty(\Ome),\; x\mapsto(\ome\mapsto\ome(x)).
\end{equation}
As already observed in \cite{junge}, this is a complete isometry of minimal spaces,
hence $\min\fX$ acts on $L^p$.

\subsection{$L^\infty$ as a minimal $p$-operator space}
We show that for a suitable measure $\mu$, $L^\infty(\mu)$ attains it minimal
$p$-operator space structure as multiplication operators on $L^p(\mu)$.

We say a measure $\mu$ is decomposable if we can write
$\mu=\sum_{\iota\in I}\mu_\iota$ where each
$\mu_\iota$ is finite, and $\mu_\iota$ and $\mu_{\iota'}$ are mutually
singular for distinct indicies.  For such measures, we have the duality $L^1(\mu)^*\cong
L^\infty(\mu)$, provided we define $L^\infty(\mu)$ to be certain equivalence classes
of locally essentially bounded functions; see \cite[p.\ 192]{folland}.
{\em We will hereafter assume $\mu$ is a
decomposable measure.}

We require a certain $p$-analogue
of a familiar result in representation theory of commutative C*-algebras
holds;  see \cite[II.1.1]{davidson}, for example, whose standard proof we modify.
We let $M_\mu:L^\infty(\mu)\to
\fB(L^p(\mu))$ be the representation given by multiplication operators.

\begin{lemma}\label{lem:p-rep}
%Let $\mu$ be a decomposable measure.
There is a locally compact space $\Ome$
such that $L^\infty(\mu)\cong\fC_b(\Ome)$ via a $*$-algebra isomorphism $f\mapsto\hat{f}$, 
a faithful Radon measure $\nu$ on $\Ome$, and a surjective isometry $U:L^p(\nu)\to L^p(\mu)$
such that $UM_\nu(\hat{f})=M_\mu(f)U$.
\end{lemma}

\begin{proof}
We first assume that $\mu$ is finite. (The proof will work for the $\sigma$-finte
case as well.)  In this case there is a norm $1$ cyclic and separating vector $\xi$ for $M_\mu$;
indeed, let $\xi$ be any fully supported norm one element.
We let $\Ome$ denote the Gelfand spectrum of $L^\infty(\mu)$ and
$f\mapsto \hat{f}$ the Gelfand transform.  We observe that
$\what{|f|^p}=|\hat{f}|^p$.  
%Indeed, for $f$ whose spectrum
%lies outside $(-\infty,0]$, the family of which is dense $L^\infty(\mu)$, we have
%$\what{|f|^p}=\bigl(\exp(\frac{p}{2}\log(f^*f))\bigr)^{\wedge}=
%\exp(\frac{p}{2}\log({\hat{f}}^*\hat{f}))=|\hat{f}|^p$.

We define $\nu$ on $\Ome$ by
\[
\int_\Ome\hat{f}\,d\nu=\int f|\xi|^p\,d\mu.
\]
Since $\xi$ is fully supported, $\nu$ is faithful.
We then define $U:\fC(\Ome)\to L^p(\mu)$ by $U\hat{f}=f\xi$.  We observe that
\[
\norm{U\hat{f}}_{L^p(\mu)}^p=\int |f|^p|\xi|^p\,d\mu
=\int_\Ome |\hat{f}|^p\,d\nu=\norm{\hat{f}}_{L^p(\nu)}^p.
\]
Since $\fC(\Ome)$ is dense in $L^p(\nu)$, and $\xi$ is a cyclic vector,
$U$ extends to a surjective isometry on $L^p(\nu)$.  Finally,
if $f,g\in L^\infty(\mu)$, then
\[
UM_\nu(\hat{f})\hat{g}=U\what{fg}=fg\xi=M_\mu(f)U\hat{g}
\]
which, again by density of $\fC(\Ome)$ in $L^p(\nu)$, shows that
$UM_\nu(\hat{f})=M_\mu(f)U$.

Now consider general decomposable $\mu=\sum_{\iota\in I}\mu_\iota$.
Let for each $\iota$, $\Ome_\iota$ denote the Gelfand spectrum of $L^\infty(\mu_\iota)$
and we have C*-isomorphisms
\[
L^\infty(\mu)\cong\ell^\infty\text{-}\bigoplus_{\iota\in I}L^\infty(\mu_\iota)
\cong\ell^\infty\text{-}\bigoplus_{\iota\in I}\fC(\Ome_\iota)
\cong\fC_b(\Ome)
\]
where $\Ome=\bigsqcup_{\iota\in I}\Ome_\iota$ is the topological coproduct.  
Let $f\mapsto\hat{f}$ denote the composite isomorphism.  We observe, moreover, that
$L^p(\mu)%=\ell^p\text{-}\bigoplus_{\iota\in I}1_{S_\iota}L^p(\mu)
\cong\ell^p\text{-}\bigoplus_{\iota\in I}L^p(\mu_\iota)$, where each
$L^p(\mu_\iota)$ is an $M_\mu$-invariant subspace.
We let $\nu_\iota$ be a measure supported on $\Ome_\iota$
given as above, and $U_\iota:L^p(\nu_\iota)\to L^p(\mu_\iota)$ the associated
surjective isometry intertwining $M_{\mu_\iota}=M_\mu |_{L^p(\mu_\iota)}$
and $M_{\nu_\iota}$.  
Then $U=\bigoplus_{\iota\in I}U_\iota$ is the desired
isometry intertwining $M_\mu$ and $M_\nu$.
\end{proof}

\begin{theorem}\label{theo:Linfty}
%Let $\mu$ be a decomposable measure.
The map $M_\mu:\min L^\infty(\mu)\to\fB(L^p(\mu))$ is
a complete isometry.
\end{theorem}

\begin{proof}
The above lemma provides a map $f\mapsto\hat{f}:L^\infty(\mu)\to\fC_b(\Ome)$,
which is a complete isometry for the minimal $p$-operator space structure on both spaces,
a faithful Radon measure $\nu$ on $\Ome$ and a surjective isometry 
$U:L^p(\nu)\to L^p(\mu)$ such that $M_\mu(f)=U^{-1}M_\nu(\hat{f})U$.  Since $M_\nu$ is 
completely isometric by Proposition \ref{prop:cbf}, we find that $M_\mu$ is a complete isometry.
\end{proof}

On the topic of $L^\infty(\mu)$,
we record the following useful result, aspects of which are folklore.
This will be used in Section\ref{ssec:Lone}.

\begin{lemma}\label{lem:commutant}
{\bf (i)} $M_\mu(L^\infty(\mu))$ is its own commutant in $\fB(L^p(\mu))$, and hence
a weak*-closed subalgebra.

{\bf (ii)}  There is a a completely contractive expectation $E:\fB(L^p(\mu))\to
M_\mu(L^\infty(\mu))$, i.e.\ $E(M_\mu(f)TM_\mu(g))=M_\mu(f)E(T)M_\mu(g)$
for $f,g$ in $L^\infty(\mu)$ and $T$ in $\fB(L^p(\mu))$.
\end{lemma}

\begin{proof}
{\bf (i)}  Let $\fF$ be the family of $\mu$-finite sets.  If $F\in\fF$ then $1_F\in L^\infty\cap L^p(\mu)$.
Fix $T$ in the comutant of $M_\mu(L^\infty(\mu))$ in $\fB(L^p(\mu))$ 
and let $h_F=T1_F$ for $F$ in $\fF$.  We observe that
for $\xi$ in $L^\infty\cap L^p(\mu)$, the space of which is dense in $L^p(\mu)$, that
$T(1_F\xi)=T(1_F)\xi=h_F\xi$, from which it easily follows that
%$\norm{h_F\xi}_p=\norm{T(1_F\xi)}_p\leq\norm{T}\norm{\xi}_p$, from which it follows that
$h_F\in L^\infty(\mu)$ with $\norm{h_F}_\infty\leq\norm{T}$.  
It is clear that $1_{F}h_{F'}=0$ and $h_F+h_{F'}=h_{F\cup F'}$, if $F\cap F'$ is $\mu$-null.
We let $\{F_\iota\}_{\iota\in I}$ be a family of sets witnessing the decomposability of $\mu$.
We observe that the net $\left(\sum_{\iota\in J}h_{F_\iota}\right)_J$,
indexed over the increasing family of finite subsets of $I$, converges weak* 
to an element $h$ of $L^\infty(\mu)$.
Indeed if $\psi\in L^1(\mu)$ then there is a $\sigma$-finite set $S$ so $1_S\psi=\psi$ and
\[
\lim_J\int\sum_{\iota\in J}h_{F_\iota}\psi\,d\mu
= \int\sum_{\iota\in I_S}h_{F_\iota}\psi\,d\mu
\]
where $I_S=\{\iota:\mu(F_\iota\cap S)>0\}$ is countable.  In particular, $h\psi=
\sum_{\iota\in I_S}h_{F_\iota}\psi$.  Now if $\xi\in L^p\cap L^\infty(\mu)$
and $\eta\in L^{p'}(\mu)$ we let $S$ be $\sigma$-finite so $1_S\xi=\xi$ and we have
\begin{align*}
\int (T\xi)\eta\,d\mu&=\int T\left(\sum_{\iota\in I_S}1_{F_\iota}\xi\right)\eta\,d\mu
=\int \sum_{\iota\in I_S}T(1_{F_\iota}\xi)\eta\,d\mu \\
&=\int \sum_{\iota\in I_S}h_{F_\iota}\xi\eta\,d\mu=\int h\xi\eta\,d\mu.
\end{align*}
Thus $T=M_\mu(h)$.  The commutant of any set in $\fB(L^p(\mu))$
is weak*-closed.

{\bf (ii)} We let $U^\infty(\mu)=\{u\in L^\infty(\mu):u^*u=1\}$.  
Let $m$ be an invariant mean on $\ell^\infty(U^\infty(\mu))$, which we may
consider, notationally, as a finitely additive measure.
We define $E$ by
\[
E(T)=\int_{U^\infty(\mu)}M_\mu(u)TM_\mu(u^*)\,dm(u)
\]
where the ``integral" is understood in the weak* sense.
Since $\mathrm{span}U^\infty(\mu)=L^\infty(\mu)$, 
it is immediate that $E$ is a contractive expectation.  If
$T\in\mat_n(\fB(L^p(\mu)))\cong\fB(\ell^p(n)\otimes^p L^p(\mu))$ we observe that
\[
E^{(n)}(T)=\int_{U^\infty(\mu)}
(I\otimes M_\mu(u))T(I\otimes M_\mu(u^*))\,dm(u)
\]
Hence $E$ is completely contractive.
\end{proof}

\section{Maximal $p$-operator spaces}\label{sec:maximal}

\subsection{Definitions and basic properties}
For a Banach space $\fX$, we consider two 
$p$-operator space structures on $\fX$, whose norms on $x$ in $\mat_n(\fX)$ are
given by
\begin{align*}
\norm{x}_{\max_{L^p}}
&=\sup\left\{\norm{\pi^{(n)}(x)}:\pi:\fX\to\fB(L^p(\phi))\text{ is a contraction, }\phi\text{ is a measure}
\right\} \\
&=\sup\left\{\norm{\pi^{(n)}(x)}:\pi:\fX\to\fB(\ell^p(m))\text{ is a contraction, }m\in\En\right\} \\
\norm{x}_{\max}
&=\sup\left\{\norm{\pi^{(n)}(x)}:\pi:\fX\to\fB(E)\text{ is a contraction, } E\in\fS\fQ_p\right\}.
\end{align*}
The equality of the two descriptions of $\norm{\cdot}_{\max_{L^p}}$ is an immediate 
consequence of Lemma \ref{lem:ellpinLp}.  It is clear that these norms give $p$-operator space structures on
$\fX$, which we call the {\em maximal structure on $L^p$} and the {\em maximal structure},
respectively.  We denote the associated operator spaces by $\max_{L^p}\fX$ and $\max\fX$.
There is an equivalent formulation of $\max\fX$ given in \cite[p.\ 95]{junge}, presented in
a local context.
It is clear that $\id:\max\fX\to\max_{L^p}\fX$ is a complete contraction.
There is no loss of generality if we replace contractions $\pi$, above, by isometries; simply
consider the isometry $\id:\fX\to\min\fX$ which acts on $L^p$ by (\ref{eq:minembedding}).

It is clear that for every operator space $\fV$ and $v$ in $\mat_n(\fV)$, that
\[
\norm{v}\leq\norm{v}_{\max}.
\] 
It is unknown to the authors whether the operator space structures $\max$ and $\max_{L^p}$
coincide on any non-trivial Banach space.  We thus use the following definition.
We say that an $p$-operator space $\fV$ is of {\em maximal type} if for $v$ in
$\mat_n(\fV)$ we have
\[
\norm{v}_{\max_{L^p}}\leq\norm{v}
\]

\begin{lemma}\label{lem:maximal}
Let $\fV$ be a $p$-operator space.  Then the following are equivalent:

{\bf (i)} $\fV$ is of maximal type;

{\bf (ii)} $\fC\fB_p(\fV,\fZ)=\fB(\fV,\fZ)$ isometrically for any $p$-operator space
$\fZ$ acting on $L^p$;

{\bf (iii)} $\fC\fB_p(\fV,\fB(\ell^p(n)))=\fB(\fV,\fB(\ell^p(n)))$ isometrically
for each $n$.
\end{lemma}

\begin{proof}
It is the case for any operator space $\fV$
that $\fC\fB_p(\fV,\fB(\ell^p(n)))\subset\fB(\fV,\fB(\ell^p(n)))$
contractively.  We obtain the converse inclusion, contractively, only for
maximal type $p$-operator spaces, by definition.  Thus (i) is equivalent to (ii).
That (ii) implies (iii) is obvious.
That (iii) implies (ii) is a consequence of Lemma \ref{lem:ellpinLp}.
\end{proof}

\begin{corollary}\label{cor:maxtp}
Let $\fV$ be a $p$-operator space.  The the following are equivalent:

{\bf (i')} $\fV$ is of maximal type;

{\bf (ii')} $\fV\what{\otimes}^p\fW=\fV\otimes^\gam\fW$, isometrically, for any
$p$-operator space $\fW$;

{\bf (iii')} $\fV\what{\otimes}^p\fN(\ell^p(m))=\fV\otimes^\gam\fN(\ell^p(m))$,
isometrically, for any $m$.
\end{corollary}

\begin{proof}
We will show that each statement (n') of the present result, is equavalent
to statment (n) of Lemma \ref{lem:maximal}

We have that $\fW^*$ represents completely isometrically on some $L^p$
by \cite[Thm.\ 4.3]{daws}.  
Hence, thanks to the well-known dual paring $\langle v\otimes w,T\rangle=
Tv(w)$  of $\fV\otimes\fW$ with $\fB(\fV,\fW^*)$
and its $p$-operator space analogue (\cite[Prop.\ 4.9]{daws}), 
if (ii) of the above lemma holds, then the $p$-operator projective and projective
tensor norms agree on $\fV\otimes\fW$.   If (ii') holds, then
statment (ii) of the above lemma holds whenever $\fZ=\fW^*$, i.e.\ for any $p$-operator 
dual space.  Hence statment (ii) holds with $\fZ^{**}$ in place of $\fZ$ we 
let $\kappa_\fZ:\fZ\to\fZ^{**}$ denote the canonical embedding and have that
$\fB(\fV,\fZ)\cong\kappa_\fZ\circ\fB(\fV,\fZ)\subset\fB(\fV,\fZ^{**})=\fC\fB_p(\fV,\fZ^{**})$ isometrically.  
If $\fZ$ acts on $L^p$ then, by \cite[Prop.\ 4.4]{daws}, $\kappa_\fZ$ is a complete
isometry so $\fC\fB_p(\fV,\fZ)\cong\kappa_\fZ\circ\fC\fB_p(\fV,\fZ)\subset
\fC\fB_p(\fV,\fZ^{**})$ isometrically, hence $\fB(\fV,\fZ)\cong\kappa_\fZ\circ\fB(\fV,\fZ)
=\kappa_\fZ\circ\fC\fB_p(\fV,\fZ)\cong\fC\fB_p(\fV,\fZ)$ isometrically, hence statment (ii)
holds generally. 

Just as above, (iii') holds if and only if (iii) of the above lemma holds.
\end{proof}

We observe that if $\fV$ and $\fW$ are each maximal type $p$-operator spaces,
then $\fV\hat{\otimes}^p\fW$ is also of maximal type.  Indeed, if $\fZ$ acts on $L^p$,
\cite[Prop.\ 4.9]{daws} provided isometric identifications
\[
\fC\fB_p(\fV\hat{\otimes}^p\fW,\fZ)
\cong\fC\fB_p(\fV,\fC\fB_p(\fW,\fZ))=\fB(\fV,\fB(\fW,\fZ))=\fB(\fV\otimes^\gam\fW,\fZ)
\]
and we appeal to statements (ii) and (ii') above.  We are unaware of whether
$\max_{L^p}\fV\hat{\otimes}^p\max_{L^p}\fW$ is completely isometric to
$\max_{L^p}(\fV\otimes^\gam\fW)$, but this does hold for $L^1$-spaces, as we
will see in Section \ref{ssec:Lone}

\subsection{Duality and quotients}

\begin{proposition}\label{prop:maximal}
{\bf (i)} If $\fV$ is a maximal type $p$-operator space then the dual structure
is minimal, i.e.\ $\fV^*=\min\fV^*$.  In particular,
$(\max\fV)^*=\min\fV^*=(\max_{L^p}\fV)^*$.

{\bf (ii)} If $\fV$ is a complete quotient of a maximal type $p$-operator space,
then $\fV$ is of maximal type.
\end{proposition}

\begin{proof}
{\bf (i)} We follow the proof form classical operator spaces -- see \cite[Cor.\ 2.8]{blecher}
or \cite[(3.3.13)]{effrosruan} -- and use Lemma \ref{lem:maximal}.  We let $\Ome$ be a dense subset
of the unit ball of $\fV$, and we have complete isometries
\[
\mat_n(\fV^*)\cong\fC\fB(\fV,\fB(\ell^p(n))=\fB(\fV,\fB(\ell^p(n))\til{\subset}
\ell^\infty(\Ome,\fB(\ell^p(n))
\]
whose composition is given by $[\psi_{ij}]\mapsto(\ome\mapsto[\psi_{ij}(\ome)])$.
By (\ref{eq:minembedding}) this is the minimal $p$-operator structure on $\fV^*$.

{\bf (ii)} If $q:\fV\to\fZ$ is a complete quotient map, and $T:\fZ\to\fB(\ell^p(n))$ is a linear contraction,
then $T\circ q:\fV\to\fB(\ell^p(n))$ is a contraction, hence a complete
contraction by (i).  Thus if $z$ is in the open unit ball of $\mat_n(\fZ)$, there
is $v$ in the open unit ball of $\mat_n(\fV)$ so $z=q^{(n)}(v)$.  Then for any linear contraction
$T:\fZ\to\fB(\ell^p(n))$ we have $\norm{T^{(n)}(z)}_{\fB(\ell^p)}=
\norm{(T\circ q)^{(n)}(v)}_{\fB(\ell^p)}<1$, so $T$ is a complete contraction.
\end{proof}

We aim obtain the dual statment to (i), above.  
We note that unlike in the $2$-operator
space setting, it is not a priori obvious that $(\min\fC(\Ome))^{**}=\min\fC(\Ome)^{**}$
completely isometrically, though we will establish this fact below. 

We require a preparatory idea
from the theory of vector measures.  For a compact Hausdorff space $\Ome$ we let
$M(\Ome)$ denote the space of complex Borel measures on $\Ome$.
Furthermore, if $E$ is a Banach space we let $M(\Ome,E)$ denote the
$E$-valued Borel measures on $\Ome$ of bounded variation.
If $E$ satisfies the Radon-Nikodym property of \cite[p.\ 61]{diesteluhl} we have
\begin{equation}\label{eq:rnp}
M(\Ome,E)=\bigcup_{\nu\in M^+(\Ome)}L^1(\nu,E)
\cong \bigcup_{\nu\in M^+(\Ome)}L^1(\nu)\otimes^\gam E
\cong M(\Ome)\otimes^\gam E
\end{equation}
where the implied isomorphism is isometric.  Indeed, if $G\in M(\Ome,E)$,
there is $\nu$ in $M^+(\Ome)$ and $g$ in $L^1(\nu,E)$ for which $G(B)=\int_B g\,d\nu$, with
$\norm{G}_{M(\Ome,E)}=|G|(B)=\norm{g}_{L^1(\nu,E)}$.
It is well-known that $L^1(\nu,E)\cong L^1(\nu)\otimes^\gam E$ isometrically.
Since, by Lebesgue decomposition, $L^1(\nu)$ is contractively compelmented
in $M(\Ome)$, we have that $L^1(\nu)\otimes^\gam E$ embeds isometrically into
$M(\Ome)\otimes^\gam E$.  Moroever, each element in $M(\Ome)\otimes^\gam E$
is an element of some $L^1(\nu)\otimes^\gam E$.  Indeed, write an element of the former
as $\sum_{k=1}^\infty\nu_k\otimes x_k$, where each $\norm{x_k}_E=1$ and
$\sum_{k=1}^\infty\norm{\nu_k}_M<\infty$.  Then let $\nu=\sum_{k=1}^\infty|\nu_k|$
and observe that each $\nu_k<\!\!<\nu$, so the element is in $L^1(\nu)\otimes^\gam E$.

\begin{theorem}\label{theo:dualmin}
If $\fW$ is a minimal operator space, then its dual operator space is maximal on $L^p$, i.e.\
$(\min\fW)^*=\max_{L^p}\fW^*$.
\end{theorem}

\begin{proof}
We begin with $\min\fC(\Ome)$ for a compact space.  From the formula
$\fV\check{\otimes}^p\fB(\ell^p(n))$ $\cong\mat_n(\fV)$ on one hand, 
and then (\ref{eq:cbf}) on the other,
we obtain for each $n$, isometric identifications
\[
\min\fC(\Ome)\check{\otimes}^p\fB(\ell^p(n))
\cong \mat_n(\min\fC(\Ome))\cong\fC(\Ome,\fB(\ell^p(n))).
\]
Taking duals, we have from \cite[Theo.\ 3.6]{anleeruan} on one hand, and 
\cite{singer1,singer2} (or see \cite[p.\ 182]{diesteluhl}) on the other, that
\[
(\min\fC(\Ome))^*\what{\otimes}^p\fN(\ell^p(n))\cong
M(\Ome,\fN(\ell^p(n))).
\]
Thanks to the fact that finite dimensional spaces enjoy the Radon-Nikodym property, we
can use (\ref{eq:rnp}) on the right hand side of the above identification to see that
\[
(\min\fC(\Ome))^*\what{\otimes}^p\fN(\ell^p(n))=M(\Ome)\otimes^\gam\fN(\ell^p(n))
\]
isometrically for each $n$.  By Corollary \ref{cor:maxtp} we see that
$M(\Ome)$, in is capacity as the dual of $\min\fC(\Ome)$, admits a maximal type
$p$-operator space structure.  Since this is a dual space, it follows \cite[Thm.\ 4.3]{daws}
that this is the maximal structure on $L^p$.

Now we consider $\min\fW\til{\subset}\min\fC(\Ome)$ where
$\Ome$ is the unit ball of $\fW^*$ with weak* topology.
Hence $\fW^*\cong \max_{L^p} M(\Ome)/\{\nu:\langle\nu,w\rangle=0\text{ for }
w\in\fW\}$ completely isometrically.  By (iii) of Proposition \ref{prop:maximal}, $\fW^*$
is of $p$-maximal type.  But by \cite[Thm.\ 4.3]{daws}, $\fW^*$ acts on some $L^p$,
hence the operator space structure is $\max_{L^p}$.
\end{proof}

We observe that it is an immediate consequence of Theorem \ref{theo:dualmin} and
Proposition \ref{prop:maximal} (i), that
\begin{equation}\label{eq:mindd}
(\min\fV)^{**}=\min\fV^{**}
\end{equation}
completely isometrically.  

As another consequence we see that for any Banach space $\fX$ and 
$\fY$
\begin{equation}\label{eq:injtenminspaces}
\min\fX\check{\otimes}^p\min\fY=\min(\fX\otimes^\lambda\fY)
\end{equation}
completely isometrically.  Indeed, we have from Lemma \ref{lem:maximal} (ii)
and (\ref{eq:minstructure}), that 
\begin{align*}
\mat_n(\fC\fB_p(\max\fX^*,\min\fY))&\cong\fC\fB_p(\max\fX^*,\mat_n(\min\fY)) \\
&=\fB(\fX^*,\mat_n(\min\fY))\cong\fB(\fX^*,\fY\otimes^\lam\fB(\ell^p(n)))
\end{align*}
isometrically. Thus the embedding of $\mat_n(\fX\otimes\fY)\cong\fX\otimes\fY\otimes\fB(\ell^p(n))$
into the space above establishes that
\[
\mat_n(\min\fX\check{\otimes}^p\min\fY)
=\fX\otimes^\lam\fY\otimes^\lam\fB(\ell^p(n))
\]
isometrically, for each $n$.  Then (\ref{eq:injtenminspaces}) follows from 
(\ref{eq:minstructure}).

%We wish to give the dual version of (\ref{eq:minembedding}), which is the $p$-version of
%\cite[3.3.2]{effrosruan}.
% and is similar to \cite[Prop.\ 3.13]{anleeruan}.
%We require, first, some refinements on duality.

%\begin{lemma}\label{lem:dual}
%Let $\fV$ be a $p$-operator space.

%{\bf (i)}  
%\end{lemma}

%\begin{proposition}\label{prop:maxellonequot}
%An $p$-operator space $\fV$ is of maximal type if and only if
%it a quotient space of $\max_{L^p}\ell^1(\Ome)$ for some set $\Ome$.
%\end{proposition}

%\begin{proof}

%\end{proof}

\subsection{$L^1$ spaces}\label{ssec:Lone}
Spaces $L^1(\mu)$, for a decomposable measure $\mu$, are the most natural
class of maximal $p$-operator spaces.

\begin{theorem}\label{theo:ellone}
The operator space structure on $L^1(\mu)$, as a subspace of
\linebreak $(\min L^\infty(\mu))^*$, is the maximal structure on $L^p$, i.e.\
$\max_{L^p}L^1(\mu)$.
\end{theorem}

\begin{proof}
We will establish that with the operator space structure given by 
$L^1(\mu)\hookrightarrow(\min L^\infty(\mu))^*$, we have
$\fC\fB_p(L^1(\mu),\fV)=\fB(L^1(\mu),\fV)$ isometrically,
for any $p$ operator space $\fV$ acting on some $L^p$.
By Lemma \ref{lem:maximal}, this implies that $L^1(\mu)$
is of maximal type.  However, since $(\min L^\infty(\mu))^*$
acts on $L^p$ (\cite[Thm.\ 4.3]{daws}), this is the $\max_{L^p}$
structure.

The assumption that $\fV$ acts
on $L^p$ implies that the embedding $\kappa_\fV:\fV\to\fV^{**}$
is a complete isometry (\cite[Prop.\ 4.4]{daws}).  We also note that 
\begin{equation}\label{eq:Lonedual}
L^1(\mu)^*\cong \min L^\infty(\mu)\text{ completely isometrically.}
\end{equation}
Indeed, as noted in \cite[Prop.\ 1.6.13]{lee}, it is sufficient, by virtue
of \cite[Prop.\ 5.5]{daws} to observe that $\min L^\infty(\mu)\cong
M_\mu(L^\infty(\mu))$ is weak* closed.  This was shown in Lemma
\ref{lem:commutant}.

We consider, first, the adjoint $S^*:\fV^*\to L^1(\Ome)^*\cong L^\infty(\Ome)$,
which is completely bounded with $\pcbnorm{S^*}=\norm{S^*}=\norm{S}$
by Proposition \ref{prop:minspace} (ii).  We then have that
$S=S^{**}\circ\kappa_{L^1(\mu)}:L^1(\mu)\to\kappa_\fV(\fV)\cong\fV$
satisfies $\pcbnorm{S}\leq\pcbnorm{S^{**}}$, which, by \cite[Lem.\ 4.5]{daws},
is no greater than $\pcbnorm{S^*}=\norm{S}$.
\end{proof}

The following is an immediate consequence of Lemma \ref{lem:commutant}
and \cite[Prop.\ 5.6]{daws}.

\begin{corollary}\label{cor:maxLonequot}
The map $\eta\otimes\xi\mapsto \eta\xi$ extends to a complete
quotient map from $\fN(L^p(\mu))=L^{p'}(\mu)\otimes^\gam L^p(\mu)$
onto $\max_{L^p}L^1(\mu)$.
\end{corollary}

We obtain the following useful tensor product formulas.  If $\fV$ is a $p$-operator space,
Corollary \ref{cor:maxtp} provides the isometric identifications
\[
{\max}_{L^p}L^1(\mu)\what{\otimes}^p\fV=L^1(\mu)\otimes^\gam\fV\cong L^1(\mu,\fV).
\]
Also we obtain a completely isometric identification
\begin{equation}\label{eq:Lonetp}
{\max}_{L^p}L^1(\mu)\what{\otimes}^p{\max}_{L^p}L^1(\nu)
={\max}_{L^p}(L^1(\mu)\otimes^\gam L^1(\nu))\cong{\max}_{L^p}L^1(\mu\times\nu).
\end{equation}
Indeed we have an isometric identification $\max_{L^p}L^1(\mu)\what{\otimes}^p\max_{L^p}L^1(\nu)
=L^1(\mu)\otimes^\gam L^1(\nu)\cong L^1(\mu\times\nu)$.  The first space has dual
$\min L^\infty(\mu)\overline{\otimes}_F \min L^\infty(\nu)$ (Fubini product) in
$\fB(L^p(\mu)\otimes^p L^p(\nu))$ by \cite[Thm.\ 6.3]{daws}, while the third has dual
$L^\infty(\mu\times\nu)$.   The latter space acts as multiplication opreators on
$L^p(\mu\times\nu)\cong L^p(\nu)\otimes^p L^p(\nu)$.  This dual identification
shows that $\min L^\infty(\mu)\overline{\otimes}_F \min L^\infty(\nu)\cong\min L^\infty(\mu\times\nu)$.
Hence (\ref{eq:Lonetp}) follows.

\medskip\noindent
{\bf Acknowledgements.}  The authors are grateful to M.\ Daws and J.-J.\ Lee 
for comments and corrections, as well as to M.\ Junge for bringing \cite{junge}
to their attention.

%\vfill

%\pagebreak

Addresses:
\linebreak
 {\sc 
 Istanbul University, Faculty of Science, Department of Mathematics, 34134 Vezneciler, Istanbul, Turkey. \\
Department of Pure Mathematics, University of Waterloo,
Waterloo, ON\quad N2L 3G1, Canada.}

\medskip
Email-adresses:
\linebreak
{\tt oztops@istanbul.edu.tr}
\linebreak {\tt nspronk@uwaterloo.ca}

\end{document}